\documentclass{amsart}

\usepackage{latexsym,enumerate}
\usepackage{amsmath,amsthm,amsopn,amstext,amscd,amsfonts,amssymb}
\usepackage[ansinew]{inputenc}
\usepackage{verbatim}
\usepackage{graphicx}
\usepackage{pstricks}
\usepackage{wasysym}
\usepackage[all]{xy}
\usepackage{epsfig} 
\usepackage{graphicx,psfrag} 
\usepackage{subfigure}
\usepackage{pstricks,pst-node}

\usepackage{multicol}
\usepackage{color}
\usepackage{colortbl}

\newcommand{\NN}{\mathbb{N}}

\newtheorem{theorem}{Theorem}[section]
\newtheorem{lemma}[theorem]{Lemma}
\newtheorem{proposition}[theorem]{Proposition}
\newtheorem{corollary}[theorem]{Corollary}
\newtheorem{definition}[theorem]{Definition}
\newtheorem{example}[theorem]{Example}
\newtheorem{remark}[theorem]{Remark}

\newcommand{\spb}[1]{\smallskip}
\newcommand{\mpb}[1]{\medskip}
\newcommand{\bpb}[1]{\bigskip}

\renewcommand{\d}{\delta}
\newcommand{\D}{\Delta}
\newcommand{\g}{\gamma}

\newcommand{\s}{\sigma}

\begin{document}
\DeclareGraphicsExtensions{.jpg,.pdf,.mps,.png}

\title{Generalized chordality, vertex separators and hyperbolicity on graphs}

\author[Alvaro Mart\'{\i}nez-P\'erez]{Alvaro Mart\'{\i}nez-P\'erez}
\address{ Facultad CC. Sociales de Talavera,
Avda. Real Fábrica de Seda, s/n. 45600 Talavera de la Reina, Toledo, Spain}
\email{alvaro.martinezperez@uclm.es}

\date{\today}


\begin{abstract} Let $G$ be a graph with the usual shortest-path metric. A graph is $\delta$-hyperbolic if for every geodesic triangle $T$, any side of $T$ is contained in a $\delta$-neighborhood of the union of the other two sides. A graph is chordal if every induced cycle has at most three edges. A vertex separator set in a graph is a set of vertices that disconnects two vertices. In this paper we study the relation between vertex separator sets, some chordality properties which are natural generalizations of being chordal and the hyperbolicity of the graph. We also give a  characterization of being quasi-isometric to a tree in terms of chordality and prove that this condition also characterizes being hyperbolic, when restricted to triangles, and having stable geodesics, when restricted to bigons.
\end{abstract}

\maketitle{}



\section{Introduction}

The theory of Gromov hyperbolic spaces was introduced by M. Gromov for the study of finitely generated groups (see \cite{Gr}). Since then, this theory has been developed from a geometric point of view to the extent of making hyperbolic spaces an important class of metric spaces to be studied on their own (see, for example, \cite{B-H,Bu-Bu,BS,G-H,Vai}). 
In the last years, Gromov hyperbolicity has been intensely studied in graphs (see \cite{BRS2,BRS,BRSV,CPRS,CRSV,CDEHV,K50,Ha,MRSV,PeRSV,PRSV,PRT1,PRT2,PT,RS,RSTY,RSVV,Si,T}). Gromov hyperbolicity, specially in graphs, has found applications in different areas such as phylogenetics (see \cite{DHH,DMT}), complex networks (see \cite{CMN,KP,Sha1,Sha2}) or the secure transmission of information and virus propagation on networks (see \cite{K21,K22}).

Given a metric space $(X,d)$,  a \emph{geodesic} from $x\in X$ to $y\in X$ is an isometry, $\gamma$, from a closed interval $[0,l]\subset \mathbb{R}$ to $X$ such that $\gamma(0)=x$, $\gamma(l)=y$. We will also call geodesic to the image of $\gamma$. $X$ is a geodesic metric space if for every $x,y\in X$ there exists a geodesic joining $x$ and $y$; any of these geodesics will be denoted as $[xy]$ although this notation is ambiguous since geodesics need not be unique.

Herein, we consider the graphs always equipped with a length metric where every edge $e\in E(G)$ has length 1. The interior points of the edges are also considered points in $G$. Then, for any pair of points in $x,y \in G$,  the distance $d(x,y)$ will be the length of the shortest path in 
$G$ joining $x$ and $y$. In this case, the graph is a geodesic metric space. Let us also assume that the graphs are connected.

There are several definitions of Gromov $\delta$-hyperbolic space  which are equivalent although the 
constant $\delta$ may appear multiplied by some constant (see \cite{BS}). We are going to use the characterization of 
Gromov hyperbolicity for geodesic metric spaces given by the Rips condition on the geodesic triangles. 
If $X$ is a geodesic metric space and $x_1,x_2,x_3\in X$, the union
of three geodesics $[x_1 x_2]$, $[x_2 x_3]$ and $[x_3 x_1]$ is called a
\emph{geodesic triangle} and will be denoted by $T=\{x_1,x_2,x_3\}$. $T$ is $\delta$-{\it thin} 
if any side of $T$ is contained in the
$\delta$-neighborhood of the union of the two other sides.  
The space $X$ is $\delta$-\emph{hyperbolic} if every geodesic triangle in $X$ is $\delta$-thin. We
denote by $\delta(X)$ the sharp hyperbolicity constant of $X$, i.e.
$\delta(X):=\inf\{\delta \, | \, \text{every triangle in $X$ is $\delta$-thin}\}.$ We say that $X$ is \emph{hyperbolic} if $X$ is
$\delta$-hyperbolic for some $\delta \geq 0$. 
A triangle with two identical vertices is called a \emph{bigon}.

A graph $G$ is \emph{chordal} if every induced cycle has at most three edges. In \cite{BKM}, the authors prove that chordal graphs are hyperbolic  giving an upper bound for the hyperbolicity constant. In \cite{WZ}, Wu and Zhang extend this result for a generalized version of chordality. They prove that $k$-chordal graphs are hyperbolic where a graph is $k$-chordal if every induced cycle has 
at most $k$ edges. In \cite{B}, the authors define the more general properties of being $(k,m)$-edge-chordal and $(k,\frac{k}{2})$-path-chordal and prove that every $(k,m)$-edge-chordal graph is hyperbolic and that every hyperbolic graph is $(k,\frac{k}{2})$-path-chordal. In \cite{MP}, we continue this work and define being $\varepsilon$-densely $(k,m)$-path-chordal and $\varepsilon$-densely $k$-path-chordal. In \cite{B} and \cite{MP}, edges where aloud to have any finite length but in this work we assume that all edges have length one. Therefore, the distinction between edge and path is unnecessary and these properties are referred as $(k,m)$-chordal and $\varepsilon$-densely 
$k$-chordal. The main results in \cite{MP} (with this simplified notation) state that 
\[(k,1)\mbox{-chordal} \Rightarrow  \varepsilon\mbox{-densely } (k,m)\mbox{-chordal} \Rightarrow \delta\mbox{-hyperbolic} \]
and
\[\delta\mbox{-hyperbolic} \Rightarrow  \varepsilon\mbox{-densely } k\mbox{-chordal} \Rightarrow k\mbox{-chordal}.\]
We also provide examples showing that for all these implications the converse is not true and we give a characterization of hyperbolicity on graphs in terms of a chordality property on the triangles.

Herein, we continue this study analysing some relations between these properties and vertex separators. There are some well known relations between chordality and vertex separators. For example, G. A. Dirac proved in \cite{D} that a graph is chordal if and only every minimal
vertex separator is complete. See also \cite{KM,SV} and \cite{BP} and the references therein. Our main results are the following.

In Section \ref{S:2}  we prove that being $(k,1)$-chordal implies that every minimal vertex separator has uniformly bounded diameter. We also obtain that, for uniform graphs, if every minimal vertex separator has uniformly bounded diameter, then the graph is $\varepsilon$-densely $(k,m)$-chordal, and therefore hyperbolic.

Section \ref{S:3} studies the relation between generalized chordality and Bottleneck Property, which is an important property on hyperbolic geodesic spaces.  J. Manning defined it in \cite{Man} and proved that a geodesic metric space satisfies (BP) if and only if it is quasi-isometric to a tree. This characterization has proved to be very useful, see for example \cite{BBF}. For some other relations with (BP) see \cite{Cas,MP12} and the references therein.

Here, we prove that a graph satisfies (BP) if and only if it is $\varepsilon$-densely $(k,m)$-chordal, providing a characterization of being quasi-isometric to a tree in terms of chordality. 
Also, the characterization of hyperbolicity from \cite{MP} is re-written obtaining that a graph is hyperbolic if and only if it is $\varepsilon$-densely $(k,m)$-chordal on the cycles that are geodesic triangles. 

Furthermore, we prove that if $G$ is a uniform graph and every minimal vertex separator has uniformly bounded diameter, then the graph satisfies (BP) and, therefore, it is quasi-isometric to a tree. Finally, we prove directly that being $(k,1)$-chordal implies $(BP)$.

In Section \ref{S:4} we generalize the concept of vertex separators defining vertex $r$-separators. It is proved that if, in a uniform graph, all minimal vertex $r$-separators have uniformly bounded diameter, then the graph is $\varepsilon$-densely $(k,m)$-chordal, and therefore, quasi-isometric to a tree. 

Section \ref{S:5} introduces neighbor separators, generalizing also vertex separators. This concept allows to characterize $(BP)$ in terms of having a neighbor-separator vertex. 

In section \ref{S:6} we define neighbor obstructors. We use them to characterize the graphs where geodesics between vertices are stable and to prove that geodesics between vertices are stable if and only if the graph is $\varepsilon$-densely $(k,m)$-chordal on the bigons defined by two vertices. We also prove that, in general, geodesics are stable if and only if the graph is $\varepsilon$-densely $(k,m)$-chordal on the bigons.

\section{Generalized chordality and minimal vertex separators}\label{S:2}

We are assuming that every path if finite and simple, this is, it has finite length and distinct vertices. By a cycle in a graph we mean a simple closed curve, this is, a path defined by a sequence of vertices which are all different except for the first one and the last one which are the same.

Let $\gamma$ be a path or a cycle. A \emph{shortcut} in  $\gamma$ is a path $\sigma$ joining two vertices $p,q$ in $\gamma$ such that $L(\sigma)<d_\gamma(p,q)$ where $L(\s)$ denotes the length of the path $\s$ and $d_\gamma$ denotes the length metric on $\gamma$. A shortcut $\sigma$ in $\gamma$ is \emph{strict} if $\sigma \cap \gamma=\{p,q\}$. In this case, we say that $p$, $q$ are 
\emph{shortcut vertices} in $\gamma$ associated to $\sigma$. A shortcut with length $k$ is called a $k$-\emph{shortcut}.

\begin{remark}\label{R:strict} Suppose $\sigma$ is a $k$-shortcut in a cycle $C$ joining two vertices, $p,q$. Then, $\sigma$ contains an strict shortcut and there are two shortcut vertices $p',q'$ such that $d_C(p,p'),d_C(q,q')< k$. 
\end{remark}

\begin{definition} A metric graph $G$ is $k$-\emph{chordal} if for any cycle $C$ in $G$ with $L(C)\geq k$ there exists a
shortcut $\sigma$ of $C$.
\end{definition}

\begin{definition} A metric graph $G$ is $(k,m)$-\emph{chordal} if for any cycle $C$ in $G$ with $L(C)\geq k$ there exists a
shortcut $\sigma$ of $C$ such that $L(\sigma)\leq m$. Notice that being chordal is equivalent to being $(4,1)$-chordal.
\end{definition}

\begin{remark}\label{k 4} Notice that in the definitions of $k$-chordal and $(k,m)$-chordal it makes no sense to consider $k\leq 3$ nor $k< 2m$. Therefore, let us assume always that $k\geq 4$ and $k\geq2m$.
\end{remark}

Given a metric space $(X,d)$ and any $\varepsilon>0$, a subset $A\subset X$ is $\varepsilon$-\emph{dense} if for every $x\in X$ there exists some $a\in A$ such that $d(a,x)<\varepsilon$.

\begin{definition} A metric graph $(G,d)$ is \emph{$\varepsilon$-densely $k$-chordal} if for every cycle $C$ with length $L(C)\geq k$, there exist strict shortcuts $\sigma_1,...,\sigma_r$  such that their associated shortcut vertices define an $\varepsilon$-dense subset in $(C,d_C)$.
\end{definition}

\begin{definition} A graph $(G,d)$ is \emph{$\varepsilon$-densely $(k,m)$-chordal} if for every cycle $C$ with length $L(C)\geq k$, there exist strict shortcuts $\sigma_1,...,\sigma_r$ with $L(\sigma_i)\leq m$ $\forall \, i$ and such that their associated shortcut vertices define an $\varepsilon$-dense subset in $(C,d_C)$.
\end{definition}

\begin{definition} A subset $S\subset V(G)$ is a \emph{separator} if $G\setminus S$ has at least two connected components. Two vertices $a$ and $b$ are \emph{separated by $S$} if they are in different connected components of $G\setminus S$. If $a$ and $b$ are two vertices separated by $S$ then $S$ is said to be an \emph{$ab$-separator}.
\end{definition}

Let us call a path joining the vertices $a,b$ an \emph{$ab$-path}. 

\begin{definition} $S$ is a \emph{minimal separator} if no proper subset of $S$ is a separator. Similarly, $S$ is a \emph{minimal $ab$-separator} if no proper subset of $S$ separates $a$ and $b$. Finally, $S$ is a \emph{minimal vertex separator} if it is a minimal separator for some pair of vertices. 
\end{definition}

Note that being a minimal vertex separator does not imply being a minimal separator. See Figure \ref{Ejp_1}.

\begin{figure}[ht]
\centering
\includegraphics[scale=0.3]{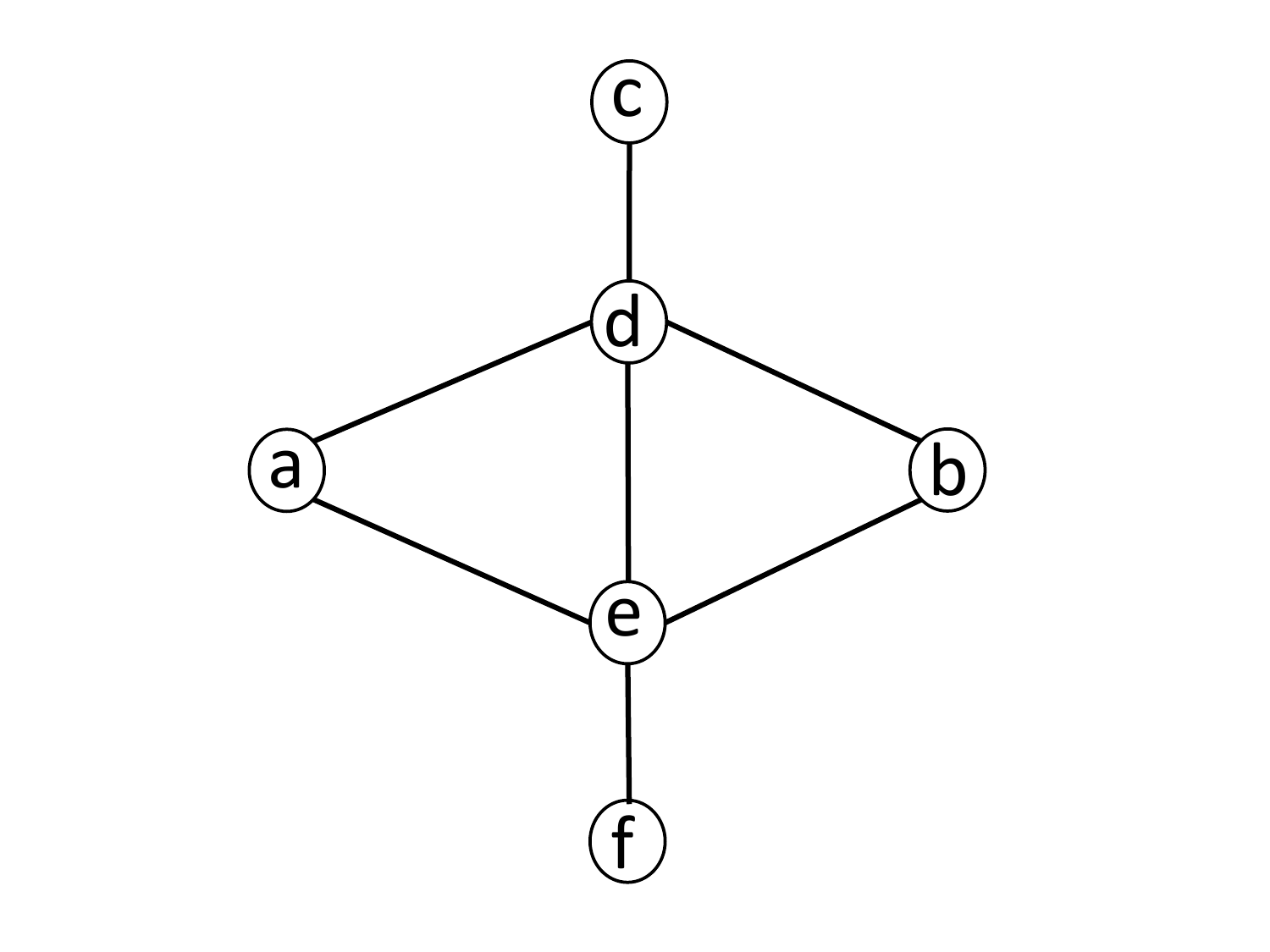}
\caption{The set $\{d,e\}$ is a minimal $ab$-separator but it is not a minimal separator.}
\label{Ejp_1}
\end{figure}

\begin{remark} Let $S$ be a minimal $ab$-separator and let $G_a$, $G_b$ be the connected components of $G\setminus S$ containing $a$ and $b$ respectively. Then, notice that every vertex $v$ in $S$ is adjacent to both $G_a$ and $G_b$. Otherwise, $S\setminus \{v\}$ is an $ab$-separator.   
\end{remark}

\begin{proposition}\label{P: chordal separator} If $G$ is $(k,1)$-chordal, then every minimal vertex separator has diameter less than $\frac{k}{2}$. 
\end{proposition}

\begin{proof} Let $S$ be a minimal $ab$-separator and suppose that $diam(S)\geq \frac{k}{2}$. Let $x,y\in S$ such that $d(x,y)\geq \frac{k}{2}$. Then, there are vertices $a_1,a_{n}$ in $G_a$ adjacent to $x$ and $y$ respectively and, since $G_a$ is connected, there is a path $\gamma_1=\{x,a_1,...,a_{n},y\}$ with $a_i\in G_a \, \forall 1\leq i \leq n$. Similarly, there exist vertices $b_1,b_{m}$ in $G_b$ adjacent to $y$ and $x$ respectively and a path $\gamma_2=\{y,b_1,...,b_{m},x\}$ with $b_i\in G_b \, \forall 1\leq i \leq m$. Moreover, let us assume that $\gamma_1,\gamma_2$ have minimal length. Then, $C=\gamma_1\cup \gamma_2$ defines a cycle in $G$ and since $d(x,y)\geq \frac{k}{2}$, $L(C)\geq k$. Then, since $G$ is $(k,1)$-chordal, there is a shortcut $\sigma$ in $C$ with $L(\sigma)=1$. However, since $S$ is an $ab$-separator, vertices in $G_a$ and $G_b$ can not be adjacent and since $\gamma_1,\gamma_2$ are supposed minimal, there is no possible 1-shortcut on $\gamma_i$ for $i=1,2$. Thus, $x,y$ need to be adjacent leading to contradiction. 
\end{proof}

The converse is not true. 

\begin{example}\label{Ex: not k-1} Consider the graph $G_0$ whose vertices are $V(G_0)=\{n\in \NN \, | \, n\geq 3\}$ and edges joining consecutive numbers. Now, let us define the graph $G$ such that for every $n\geq 3$, there is cycle $C_n$ whose vertices are all adjacent to the vertex $n$ in $G_0$. See Figure \ref{Not k-1}.

\begin{figure}[ht]
\centering
\includegraphics[scale=0.3]{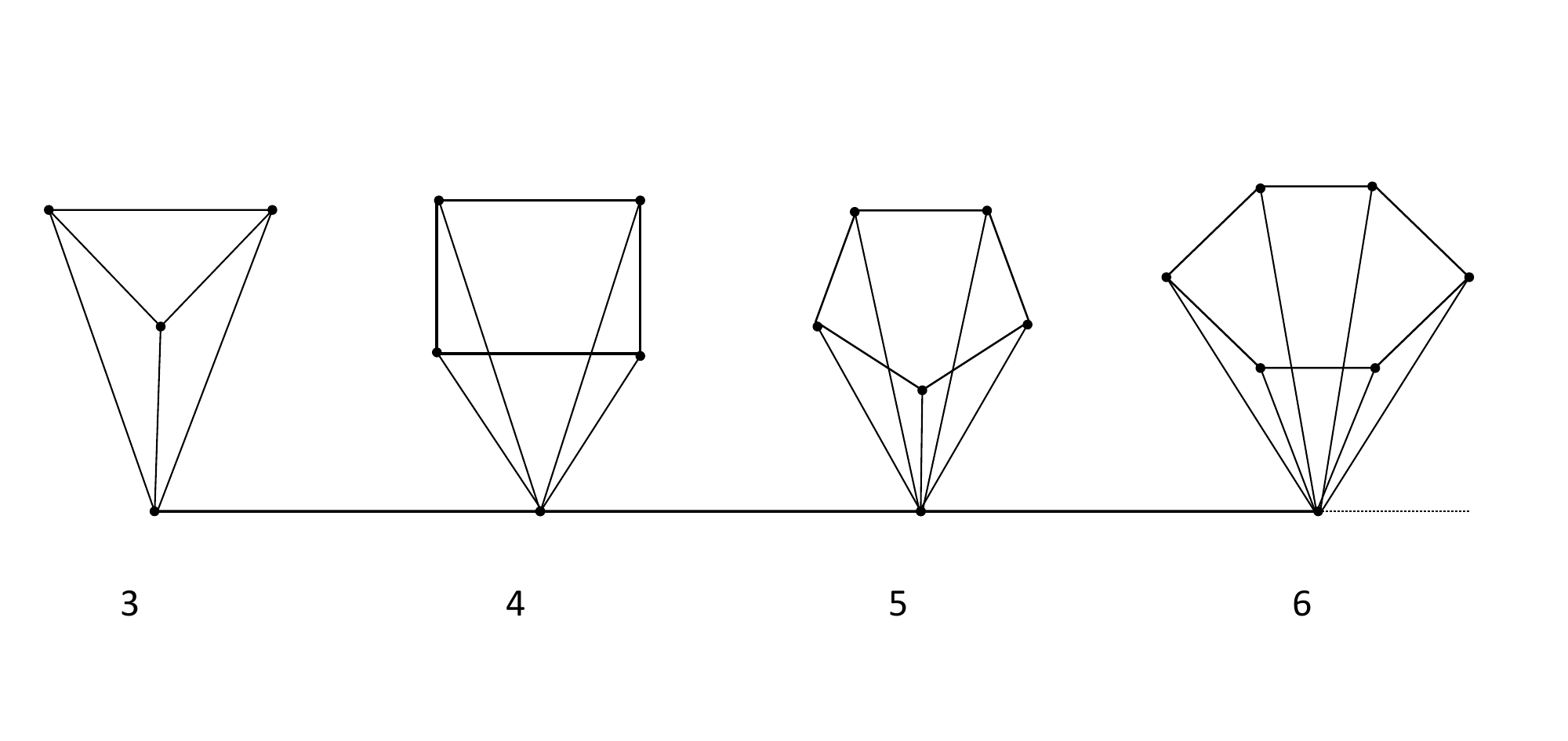}
\caption{Every minimal vertex separator has diameter at most 2, but the graph is not $(k,1)$-chordal for any $k>0$.}
\label{Not k-1}
\end{figure}

It is trivial to check that $G$ is not $(k,1)$-chordal for any $k>0$ since the cycles $C_n$ have no 1-shortcut in $G$. 

Let us see that every minimal vertex separator has diameter at most 2. Consider any pair of non-adjacent vertices $a,b$ in $G$.   

If $a,b\in C_n$ for some $n$, then every vertex separator $S$ must contain the vertex $n$ and at least two vertices $x_1,x_2$ in $C_n$. If $S$ is minimal, then $S=\{n,x_1,x_2\}$ and $diam(S)=2$.

If $a,b\notin C_n$ for any $n$, then the geodesic $[ab]$ is contained in $G_0$. Therefore, any $ab$-separator must contain some vertex $m\in [ab]$ and $m$ separates $a$ and $b$. Thus, if $S$ is minimal, then $S$ is just a vertex and $diam(S)=0$.
\end{example}

\begin{remark}\label{R:separator} Given two vertices $a,b$, a path $\gamma$ joining them and a vertex $v\in \gamma$ distinct from $a,b$, there may not exist a minimal $ab$-separator containing $\{v\}$. Consider, for example four vertices $x_0,x_1,x_2,x_3$ with edges $x_{i-1}x_i$ for every $1\leq i \leq 3$ and an edge $x_0x_2$. Then, there is no minimal $x_0x_3$-separator containing $x_1$.
\end{remark}

Given a graph $G$ and a subgraph, $A\subset G$, let us denote $V(A)$ the vertices in $A$.  

\begin{definition} A graph $\Gamma$ is said to be $\mu$-\emph{uniform} if each vertex $p$ of $V$ has at most $\mu$ neighbors, i.e.,
\[\sup\big\{|N(p)| \,  \big| \,\, p\in V(\Gamma)\big\}\leq \mu.  \]
If a graph $\Gamma$ is $\mu$-uniform for some constant $\mu$ we say that $\Gamma$ is \emph{uniform}.
\end{definition}

For any vertex $v\in V(G)$ and any constant $\varepsilon>0$, let us denote:
\[S_\varepsilon(v):=\{w\in V(G) \, | \, d(v,w)=\varepsilon\}.\]
\[B_\varepsilon(v):=\{w\in V(G) \, | \, d(v,w)<\varepsilon\}.\]
\[N_\varepsilon(v):=\{w\in V(G) \, | \, d(v,w)\leq \varepsilon\}.\]

\begin{lemma}\label{L:separator} Let $G$ be a uniform graph. Given two vertices $a,b$, a geodesic 
$[ab]$ joining them and a vertex $v_0\in [ab]$ distinct from $a,b$, then there is a minimal $ab$-separator containing $\{v_0\}$.
\end{lemma}

\begin{proof} Suppose any geodesic $[ab]$ and $v_0\in [ab]$ with $0<d(a,v_0)<d(a,b)$ and define $\varepsilon=d(a,v_0)$. Since $G$ is uniform, for every vertex $v\in G$ the set $S_0:=S(v,\varepsilon)$ is finite for every $\varepsilon\in \NN$. It is immediate to check that $S_0$ is an $ab$-separator and $[ab]\cap S_0=\{v_0\}$. Since $S_0$ is finite, then there is a minimal subset $S\subset S_0$ which is also an $ab$-separator. Finally, since $[ab]\cap S_0=\{v_0\}$, $v_0\in S$.
\end{proof}

Let us recall that a graph $\Gamma$ is \emph{countable} if $|V(\Gamma)|\leq \aleph_0$, i.e. if it has a countable number of vertices.

\begin{remark} \label{R: choice} In the case of countable graphs and using the Axiom of Choice, Lemma \ref{L:separator} can be slightly improved. See Lemma \ref{L:separator 2} below.
\end{remark}

\begin{lemma}\label{L:separator 2} Let $G$ be a uniform countable graph. Given two vertices $a,b$, a path $\gamma_0$ joining them and a vertex $v_0\in \gamma_0$ distinct from $a,b$, then either there is a 1-shortcut in $\gamma_0$ or there is a minimal $ab$-separator containing $\{v_0\}$.
\end{lemma}

\begin{proof} If there is no $ab$-path in $G\setminus \{v_0\}$ it suffices to consider $S:=\{v_0\}$.

If there is an $ab$-path $\gamma_1$ in $G\setminus \{v_0\}$ such that $V(\gamma_1)\subset V(\gamma_0)$, then there is a 1-shortcut in $\gamma_0$.

Thus, let us suppose that every  $ab$-path $\gamma$ in $G\setminus \{v_0\}$ contains a vertex which is not in $\gamma_0$ and that there is at least one of these $ab$-paths.

Since $|V(G)|$ is countable and $G$ is uniform there exist at most $\aleph_0^k$ $ab$-paths of length $k$. Then,  there exist at most a countable number (a countable union of countable sets) of $ab$-paths, $\{\g_i\}_{i\in I\subset \NN}$ in $G\setminus \{v_0\}$ where $I=\{1,\dots, m\}$ if there exist exactly $m$ such paths or $I=\NN$ if the number of those paths is not finite. 

For every $i\in I$, consider some vertex $x_i$ in $V(\gamma_i)\setminus V(\g_0)$ and let $X=\{x_i\}_{i\in I}$.

Now let $S_0:=X$ and for every $0<i\in I$ define:
$$S_i = \left\{ 
\begin{array}{c}
S_{i-1}\setminus \{x_i\} \quad  \mbox{ if } V(\gamma_j)\cap \Big(S_{i-1}\setminus \{x_i\}\Big)\neq \emptyset \mbox{ for every } j\leq i, \\ 
\ S_{i-1} \qquad \quad \mbox{ if } V(\gamma_j)\cap \Big(S_{i-1}\setminus \{x_i\}\Big) =\emptyset \mbox{ for some } j\leq i. 
\end{array}\right. $$

Notice that for every $i$, $S_i\subset S_{i-1}$ and let $S:=\cap_{i\in I}S_i$.

\smallskip

Claim: $S$ is a minimal $ab$-separator containing $v_0$.

\smallskip

First, let us see that $S$ is an $ab$-separator. Consider any $ab$-path, $\gamma_j$. 
Suppose $V(\g_j)\cap X=\{x_{j_1},x_{j_2},...,x_{j_k}\}$ and assume $j_l<j_k$ for every $l<k$. Then, it is trivial to check that there exist some vertex $x_{j_r}\in S_{j_k}\cap V(\g_j)$ and, by construction, $x_{j_r}\in S$. 

To check that $S$ is minimal, first notice that, since $x_i\notin V(\g_0)$ for every $i\in I$, $V(\gamma_0)\cap (S\backslash \{v_0\})=\emptyset$ and $S\backslash \{v_0\}$ is not an $ab$-separator. Now, suppose that there is some vertex $x_j \in S$ with $j\in I$ such that $S\setminus \{x_j\}$ is also an $ab$-separator. Since  $x_j \in S\subset S_j$, there is some $k\leq j$ such that 
$V(\g_{k})\cap \Big(S_{j-1}\setminus \{x_j\}\Big)= \emptyset$ and, in particular, 
$V(\g_{k})\cap \Big(S\setminus \{x_j\}\Big)= \emptyset$ leading to contradiction.

Thus, $S$ is a minimal $ab$-separator containing $v_0$. 
\end{proof}

\begin{theorem}\label{T: separator} Let $G$ be a uniform graph. If every minimal vertex separator in $G$ has diameter at most $m$, then $G$ is 
$(m+\epsilon)$-densely $(4m,2m-1)$-chordal for any $\epsilon>\frac12$. 
\end{theorem}

\begin{proof} Let $C$ be any cycle with $L(C)\geq 4m$.
Let $v$ be any vertex in $C$ and let $a,b$ be the two vertices in $C$ such that 
$d_C(a,v) =m=d_C(v,b)$. Let $\gamma_0$ be the $ab$-path 
in $C$ containing $v$. Then by Lemma \ref{L:separator}, either there is a shortcut in $\gamma_0$ or there is a minimal $ab$-separator containing $v$. 

If there is a shortcut in $\gamma_0$ then it has length at most $2m-1$. Therefore, it defines a shortcut in $C$ with an associated shortcut vertex $v'$ such that $d_C(v,v')\leq m$. 
Suppose, otherwise, that $S$ is a minimal $ab$-separator containing $v$. By hypothesis, $diam(S)\leq m$. Let $\gamma_1$ 
be the $ab$-path in $C$ not containing $v$. Since $S$ is an $ab$-separator, there is some vertex $w\in S\cap V(\gamma_1)$ and $d(v,w)\leq m <d_C(v,w)$. Hence, there is an $m$-shortcut in $C$ joining $v$ to $w$ and, by Remark \ref{R:strict}, an associated shortcut vertex $v'$ such that $d_C(v,v')<m$.

Thus, for every vertex $v$, there is a shortcut vertex $v'$ such that $d_C(v,v')\leq m$ and therefore, shortcut vertices define a $(m+\epsilon)$-dense subset in $C$ for any $\varepsilon >\frac12$.
\end{proof}

If the graph is countable, then we can improve quantitatively this result.

\begin{theorem}\label{T: separator 2} Let $G$ be a uniform countable graph. If every minimal vertex separator in $G$ has diameter at most $m$, then $G$ is 
$(m+\epsilon)$-densely $(2m+2,m)$-chordal for any $\epsilon>\frac12$. 
\end{theorem}

\begin{proof} Let $C$ be any cycle with $L(C)\geq 2m+2$.
Let $v$ be any vertex in $C$ and let $a,b$ be the two vertices in $C$ such that 
$d_C(a,v) =m=d_C(v,b)$. Let $\gamma_0$ be the $ab$-path 
in $C$ containing $v$. Then by Lemma \ref{L:separator 2}, either there is a 1-shortcut in $\gamma_0$ or there is a minimal $ab$-separator containing $v$. 

If there is a 1-shortcut in $\gamma_0$ then, in particular, there is an associated shortcut vertex $v'$ such that $d_C(v,v')\leq m$. 
Suppose, otherwise, that $S$ is a minimal $ab$-separator containing $v$. By hypothesis, $diam(S)\leq m$. Let $\gamma_1$ 
be the $ab$-path in $C$ not containing $v$. Since $S$ is an $ab$-separator, there is some vertex $w\in S\cap V(\gamma_1)$ and $d(v,w)\leq m <d_C(v,w)$. Hence, there is an $m$-shortcut in $C$ joining $v$ to $w$ and, by Remark \ref{R:strict}, an associated shortcut vertex $v'$ such that $d_C(v,v')< m$.

Thus, for every vertex $v$, there is a shortcut vertex $v'$ such that $d_C(v,v')\leq m$ and therefore, shortcut vertices define a $(m+\epsilon)$-dense subset in $C$ for any $\varepsilon >\frac12$.
\end{proof}

Let us recall the following result:

\begin{theorem}\cite[Theorem 4]{MP}\label{Th: hyperbolic} If $G$ is $\varepsilon$-densely $(k,m)$-chordal, then $G$ is hyperbolic. Moreover,  $\delta(G)\leq \max\{\frac{k}{4}, \varepsilon+m\}$.
\end{theorem}

Therefore, from theorems  \ref{T: separator},  \ref{T: separator 2} and \ref{Th: hyperbolic} we obtain:

\begin{corollary}\label{Cor: hyperbolic} Let $G$ be a uniform graph. If every minimal vertex separator in $G$ has diameter at most $m$, then $G$ is hyperbolic. Moreover, $\delta(G)\leq 3m-\frac{1}{2}$.
\end{corollary}

\begin{corollary}\label{Cor: hyperbolic 2} Let $G$ be a uniform countable graph. If every minimal vertex separator in $G$ has diameter at most $m$, then $G$ is hyperbolic. Moreover, $\delta(G)\leq 2m +\frac{1}{2}$.
\end{corollary}

\section{Bottleneck Property}\label{S:3}

Let us recall the following definition from \cite{Man}:

\begin{definition} A geodesic metric space $(X,d)$ satisfies \emph{Bottleneck Property} (BP) if there exists some constant $\Delta>0$ so that given any two distinct points $ x, y\in X$ and a midpoint $z$ such that $d(x,z)=d(z,y)=\frac12 d(x,y)$, then every $xy$-path intersects $N_\D(z)$.
\end{definition}

\begin{remark} This definition although is not exactly the same, it is equivalent to Manning's. In the original definition J. Manning asked only for the existence of such a midpoint for any pair of points $x,y$. However, by Theorem \ref{Th: quasi-tree} below, (BP) implies that the space is quasi-isometric to a tree and therefore $\delta$-hyperbolic. Hence, it is an easy exercise in hyperbolic spaces to prove that if there is always a midpoint $z$ such such that every $xy$-path intersects 
$N_{\D}(z)$, then this condition holds in general for any midpoint, possibly with a different constant depending only on $\D$ and $\d$. See, for example, Chapter 2, Proposition 25 in \cite{G-H}. 
\end{remark}

\begin{definition} A graph $G$ satisfies (BP) on the vertices if there exists some constant $\Delta'>0$ so that given any two distinct vertices $ v, w\in V(G)$ and a midpoint $c$ such that $d(v,c)=d(c,w)=\frac12 d(v,w)$, then every $vw$-path intersects $N_{\Delta'}(c)$.
\end{definition}

\begin{proposition}\label{P: BP graphs} A 
graph $G$ satisfies (BP) if and only if it satisfies (BP) on the vertices. Moreover, if $G$  satisfies (BP) on the vertices with constant $\D'$, it satisfies (BP) with $\D=\Delta'+\frac32$. 
\end{proposition}

\begin{proof} The only if condition is trivial. Let us see that it suffices to check the property on the pairs of vertices. 

Consider any pair of points $x,y\in G$ and let $z$ be a midpoint of a geodesic $[xy]$. 

If $d(x,y)\leq 2$, then (BP) is trivial with $\Delta=1$.

Suppose $d(x,y)> 2$. Then, the geodesic $[xy]$ is a path $xv_1\cup v_1v_2\cup \cdots \cup v_ky$ with $v_1,\dots ,v_k\in V(G)$ and $k\geq 2$. Let $v=x$ if $x$ is a vertex and $v=v_1$ otherwise, and let $w=y$ if $y$ is a vertex and $w=v_k$ otherwise. Then, there is a geodesic $[vw]\subset [xy]$ 
(possibly equal), and its midpoint, 
$c$, satisfies that $d(c,z)\leq \frac12$.

Consider any $xy$-path $\g$ and let us define a $vw$-path $\g'$ as follows: 
First, if  $v\in \g$, let $\g_0:=\g\setminus [xv)$ and if  $v\notin \g$, let  $\g_0:=[vx]\cup \g$.  
Then, if $y\neq w$ and $w\in \g$, let $\g':=\g_0\setminus [yw)$ and if $y\neq w$ and $w\notin \g$, let $\g':=[wy]\cup \g_0$.  

By hypothesis,  $\g'$ passes through $N_{\Delta'}(c)$. Since $d(a,v),d(b,w)\leq 1$ and $d(c,z)\leq \frac12$ it is immediate to check that $\g$ passes through $N_{\Delta'+\frac32}(z)$.
\end{proof}

A map between metric spaces, $f:(X,d_X)\to (Y,d_Y)$, is
said to be a \emph{quasi-isometric embedding} if there are constants $\lambda
\geq 1$ and $C>0$ such that $\forall x,x'\in X$,
$$\frac{1}{\lambda}d_X(x,x') -C \leq d_Y(f(x),f(x'))\leq \lambda d_X(x,x')+C.$$ 
If if there is a constant $D>0$ such that
$\forall y\in Y$, $d(y,f(X))\leq D$, then $f$ is a
\emph{quasi-isometry} and $X,Y$ are \emph{quasi-isometric}.

\begin{theorem}\cite[Theorem 4.6]{Man}\label{Th: quasi-tree} A geodesic metric space $(X,d)$ is quasi-isometric to a tree if and only if it satisfies (BP).
\end{theorem}

\begin{theorem}\label{Th: ch-BP} A graph $G$ satisfies (BP) if and only if it is $\varepsilon$-densely $(k,m)$-chordal. 
\end{theorem}

\begin{proof} Suppose that $G$ satisfies (BP) with parameter $\D$ and consider any cycle $C$ with  
$L(C)\geq 2\D+4$. Consider any vertex $x\in C$ and the two vertices $a,b$ such that $d_C(a,x)=d_C(x,b)=\D+1$. Thus, $C$ defines two $ab$-paths, $\g_1,\g_2$. Let us assume that $x\in \g_1$. If $\g_1$ is not geodesic, then there is a shortcut with length at most $2\D+1$ and a shortcut vertex in $N_{\D+1}(x)$. Otherwise, since $G$ satisfies (BP) with parameter $\D$, there is a vertex $y$ in $\g_2$ such that $d(x,y)\leq \D$.  Since $d_C(x,y)>\D$ and by Remark \ref{R:strict},  there is a shortcut vertex $z$ such that $d_C(x,z)<\D$. Therefore, $G$ is $(\D+1+\varepsilon)$-densely $(4\D+4,2\D+1)$-chordal for any $\varepsilon> \frac12$. 

Suppose that $G$ is $\varepsilon$-densely $(k,m)$-chordal and it does not satisfy (BP) with parameter $\D=\max\{\frac{k}{4},\varepsilon+m\}$. Then, there are two points, $a,b$, a geodesic $[ab]$ with midpoint $c$ and a path $\g$ such that $\g\cap N_\D(c)=\emptyset$. Then, it is immediate to check that there exist two points $a',b'\in \g\cap [ab]$ such that the restriction of $[ab]$, $[a'b']$, and the restriction of $\g$, $\g'$, joining $a'$ to $b'$ define a cycle $C$ with $L(C)> k$. Since $G$ is $\varepsilon$-densely $(k,m)$-chordal, there is a strict shortcut $\s$ with $L(\s)\leq m$ with an associated shortcut vertex $w$ such that $d_C(c,w)<\varepsilon<\D$. Therefore, $w\in [a'b']$ and, since $[a'b']$ is geodesic, the shortcut must join $w$ to a vertex $z$ in $\g'\subset \g$. Hence, $d(z,c)< \varepsilon +m$ and $\g\cap N_\D(c)\neq \emptyset$ leading to contradiction.
\end{proof}

\begin{corollary}\label{C: ch-BP} A graph $G$ is quasi-isometric to a tree if and only if it is 
$\varepsilon$-densely $(k,m)$-chordal.
\end{corollary}

\begin{definition} Given any family $\mathcal{F}$ of cycles, a metric graph $(G,d)$ is \emph{$\varepsilon$-densely $(k,m)$-chordal on $\mathcal{F}$} if for every $C\in \mathcal{F}$ with length $L(C)\geq k$, there exist strict shortcuts $\sigma_1,...,\sigma_r$ with $L(\s_i)\leq m$ $\forall \, i$ and such that their associated shortcut vertices define an $\varepsilon$-dense subset in $(C,d_C)$.
\end{definition}

Let us recall the following:

\begin{lemma} \cite[Lemma 2.1]{RT1} \label{L: triangle cycle} Let $X$ be a geodesic metric space. If every geodesic triangle in $X$ which is a cycle is $\d$-thin, then $X$ is $\d$-hyperbolic.
\end{lemma}

Let $\mathcal{T}$ be the family of cycles that are geodesic triangles. 
It is immediate to check that, using Lemma \ref{L: triangle cycle}, the proof of Theorem 13 in \cite{MP} can be trivially re-written (we include it for completeness) to obtain the following:

\begin{theorem}\label{Th: carac2} $G$ is $\delta$-hyperbolic if and only if $G$ is $\varepsilon$-densely $(k,m)$-chordal on $\mathcal{T}$.
\end{theorem}

\begin{proof} Suppose that $G$ is $\varepsilon$-densely $(k,m)$-path-chordal on $\mathcal{T}$. Let us see that $\delta(G)\leq \max\{\frac{k}{4},\varepsilon+m\}$. Consider any cycle which is a  geodesic triangle $T=\{x,y,z\}$. If $L(T)<k$, it follows that every side of the triangle has length at most $\frac{k}{2}$. Therefore, the hyperbolic constant is at most $\frac{k}{4}$. Then, let $L(T)\geq k$ and let us prove that $T$ is $(\varepsilon+m)$-thin. Consider any point $p\in T$ and let us assume that $p\in [xy]$. If $d(p,x)< \varepsilon+m$ or $d(p,y)< \varepsilon+m$, we are done. Otherwise, there is a shortcut vertex $x_i$ such that $d(x_i,p)<\varepsilon$ and a shortcut $\sigma_i$, with $x_i\in \sigma_i$ and $L(\sigma_i)\leq m$. Since $[xy]$ is a geodesic, $\sigma_i$ does not connect two points in $[xy]$ and $d(p, [xz]\cup [yz])<\varepsilon+m$. Then, by Lemma 
\ref{L: triangle cycle},  $\delta(G)\leq \max\{\frac{k}{4},\varepsilon+m\}$.

Suppose that $G$ is $\delta$-hyperbolic and consider any cycle which is a geodesic triangle $T=\{x,y,z\}$ with $L(T)\geq 9\delta$. Let $p\in T$ and let us assume, with no loss of generality, that $p\in [xy]$. Since $G$ is $\delta$-hyperbolic, $d(p, [xz]\cup [yz])\leq \delta$. If $d(p,x),d(p,y)>\delta$, then there is a path $\gamma$ with $L(\gamma)\leq \delta$ joining $p$ to $[xz]\cup [yz]$. In particular, there is a shortcut $\sigma\subset \gamma$ with $L(\sigma)\leq L(\gamma)\leq \delta$ joining some shortcut vertex $p'\in [xy]$ with $d(p,p')< \delta$ to $[xz]\cup [yz]$. Therefore, if $L([xy])>2\delta$, for every point $q\in [xy]$ there is a shortcut vertex $q'\in [xy]$ such that 
$d_T(q,q')<2\d+1$ associated to a shortcut with length at most $\delta$. Since $L(T)\geq 9\delta$, by triangle inequality, there is at most one side of the triangle with length at most $2\delta$. Then, for every point $p$ in the triangle there is a shortcut vertex $p'$ such that $d_T(p,p')<3\d+1$ associated to a shortcut with length at most $\delta$. Thus, it suffices to consider $\varepsilon=3\delta+1$,  $k=9\delta$ and $m=\delta$.
\end{proof}

\begin{remark} Notice that in Corollary \ref{C: ch-BP} we obtain that a graph $G$ is quasi-isometric to a tree if and only if all the cycles satisfy certain property and Theorem \ref{Th: carac2}  states that the same property, restricted to the cycles which are geodesic triangles, characterizes being hyperbolic. 
\end{remark}

The following theorem can be also obtained as a corollary of theorems \ref{T: separator} and \ref{Th: ch-BP}. However, the direct proof provides a better bound for the parameter $\D$.

\begin{theorem}\label{Th: separator BP} Given a uniform graph $G$, if every minimal vertex separator has diameter at most $m$ then $G$ satisfies (BP) (i.e., $G$ is quasi-isometric to a tree). Moreover, it suffices to take $\D=m+2$.
\end{theorem}

\begin{proof} If $m=0$ it is trivial to see that $G$ is a tree and it satisfies (BP) with $\D=0$.

Assume $m\geq 1$. By Proposition \ref{P: BP graphs}, it suffices to check the property for pairs of vertices. Thus, consider any pair of vertices $a,b\in V(G)$ and let $c$ be a midpoint of a geodesic $[ab]$. 

If $d(a,b)\leq 2$, then (BP) is trivial with $\Delta'=1$.

Suppose $d(a,b)\geq 3$. Then, there is some vertex $v_0$ in the interior of $[ab]$ with $d(v_0,c)\leq \frac12$. By Lemma \ref{L:separator}, since $[ab]$ is a geodesic, there exist a minimal $ab$-separator $S$ containing $v_0$. Thus, every $ab$-path contains a vertex in $S$ and since $diam(S)\leq m$, every $ab$-path passes through $N_m(v_0)\subset N_{m+\frac12}(c)$. Hence, (BP) is satisfied on the vertices with $\D'=m+\frac{1}{2}$ and, by Proposition \ref{P: BP graphs}, $G$ satisfies (BP) with $\D=m+2$.
\end{proof}

The following example shows that converse is not true.

\begin{example} Let $G$ be the graph whose vertices are all the pairs $(a,b)$ with either $a\in \NN$ and $b=0$ or $4n+1\leq a \leq 4n+3$ and  $1 \leq b \leq n$ for every $n\in \NN$, and such that $(a,b)$ is adjacent to $(a',b')$ if and only if either $b=b'$ and $|a'-a|=1$ or $a=a'$ and $|b'-b|=1$. See Figure \ref{BP}.

Now, notice that $S_n=\{4n+2,j\}_{0\leq j\leq n}$ defines a minimal $(4n,0)(4n+4,0)$-separator with diameter $n$ for every $n\in \NN$. Therefore, $G$ has minimal $ab$-separators arbitrarily big. However, to see that $G$ satisfies (BP), consider the map $f: V(G) \to V(G)$ such that $f(i,j)=(4n+2,j)$ for every $4n+1\leq i \leq 4n+3$ and $1\leq j \leq n$ and the identity on the rest of the vertices. It is trivial to check that $f$ extends to a $(1,2)$-quasi-isometry on $G$ where the image is a tree. Therefore, $G$ is quasi-isometric to a tree and satisfies (BP) (and it is $\varepsilon$-densely $(k,m)$-chordal). 

\begin{figure}[ht]
\centering
\includegraphics[scale=0.4]{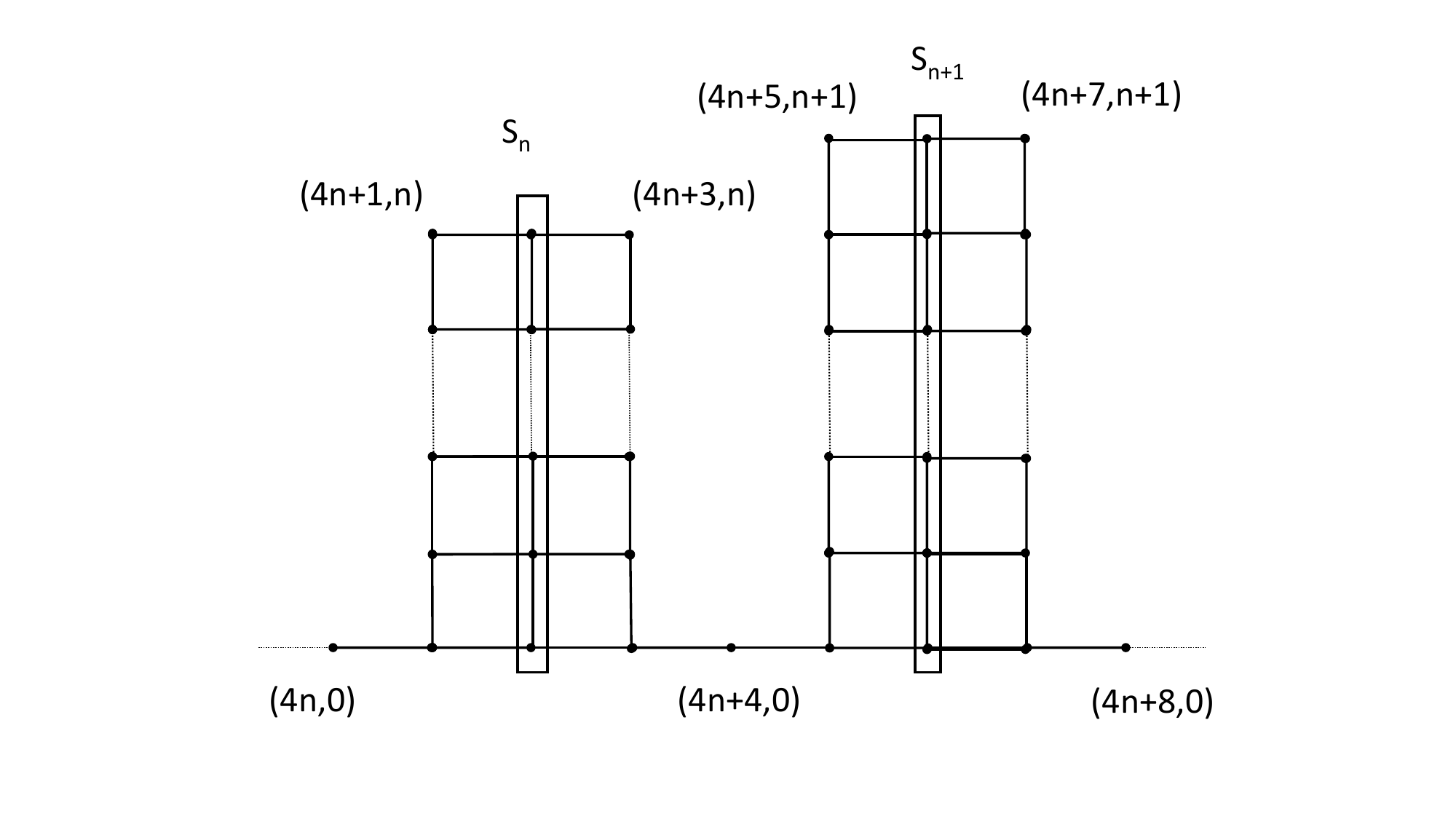}
\caption{Satisfying Bottleneck Property does not imply the existence of minimal vertex separators   with uniformly bounded diameter.}
\label{BP}
\end{figure}
\end{example}

\begin{remark} In the case of uniform graphs, the following theorem can also be obtained as a corollary of Proposition \ref{P: chordal separator} and Theorem \ref{Th: separator BP}. Also, it follows from Theorem 3 in \cite{MP} and Theorem \ref{Th: ch-BP}. However, the direct proof provides a better bound for the parameter.
\end{remark}

\begin{theorem}\label{T: chordal BP} If $G$ is $(k,1)$-chordal, then $G$ satisfies (BP). Moreover, it suffices to take $\D=\frac{k}{4}+\frac52$.
\end{theorem}

\begin{proof} Consider any pair of vertices $a,b$, any geodesic $[ab]$ in $G$ and the midpoint $c$ in $[ab]$. If $d(a,b)\leq \frac{k}{2}+2$, then (BP) is trivially satisfied for $\D'=\frac{k}{4}+1$. Suppose $d(a,b)> \frac{k}{2}+2$ and that there is an $ab$-path $\g$ not intersecting $N_{k/4+1}(c)$. 
Let $a'\in [ac]\subset [ab]$ and $b'\in [cb]\subset [ab]$ such that $d(a',c)=d(c,b')=\frac{k}{4}$. Then, since $\g$ does not intersect $N_{k/4+1}(c)$, there is a cycle $C$ contained in $[ab]\cup \g$ such that $[a'b']\subset C$ and $L(C)\geq k$.

Claim: there is a 1-shortcut in $C$ joining a vertex in the interior of $[a'b']$ to a vertex in 
$\g$. Since $G$ is $(k,1)$-chordal, there is a 1-shortcut, $\s_1$, in $C$. If $\s_1$ joins a vertex in the interior of $[a'b']$ to a vertex in $\g$, we are done. Otherwise,  we obtain a new cycle, 
$C_1$, such that $[a'b']\subset C_1$ and, therefore, $L(C_1)\geq k$. Repeating the process we finally obtain a 1-shortcut joining a vertex $z_1$ in the interior of $[a'b']$ to a vertex $z_2$ in $\g$. 

Therefore, $d(c,z_2)\leq d(c,z_1)+1< \frac{k}{4}+1$ and $z_2\in N_{k/4+1}(c)$ leading to contradiction.

Thus, $G$ satisfies (BP) on the vertices with $\D=\frac{k}{4}+1$ and, by Proposition \ref{P: BP graphs}, $G$ satisfies (BP) with $\D=\frac{k}{4}+\frac52$.
\end{proof}

\begin{corollary}\label{C: chordal qi} If $G$ is $(k,1)$-chordal, then $G$ is quasi-isometric to a tree.
\end{corollary}

\begin{remark} Corollary \ref{C: chordal qi} follows also from Proposition \ref{P: chordal separator} and Corollary \ref{Cor: hyperbolic} in the case of uniform graphs.
\end{remark}

\begin{remark} The converse to Theorem \ref{T: chordal BP} or Corollary \ref{C: chordal qi} is not true. It is immediate to check that the graph from Example \ref{Ex: not k-1} is quasi-isometric to a tree through the map sending every cycle $C_n$ to the vertex $n$.
\end{remark}

\section{Minimal vertex $r$-separators}\label{S:4}

\begin{definition} Given $r\in \mathbb{N}$, two vertices $a$ and $b$ are \emph{$r$-separated} by a subset $S\subset V(G)$ if considering the connected components of $G\setminus S$, $G_a$ and $G_b$ containing $a$ and $b$ respectively, for every pair of vertices $v\in G_a$ and $w\in G_b$, 
$d(v,w)>r$. If $a$ and $b$ are two vertices $r$-separated by $S$ then $S$ is said to be an \emph{ $ab$-$r$-separator}.
\end{definition}

\begin{remark} Notice that separated means 1-separated. 
\end{remark}

\begin{definition}  $S$ is a \emph{minimal $ab$-$r$-separator} if no proper subset of $S$ $r$-separates $a$ and $b$. Finally, $S$ is a \emph{minimal vertex $r$-separator} if it is a minimal $r$-separator for some pair of vertices.
\end{definition}

\begin{remark}\label{R: r-separator} Given any minimal $ab$-$r$-separator $S$, every vertex in $S$ is either adjacent to $G_a$ or $G_b$. Moreover, if $r\geq 2$, then there are two disjoint subsets $S_a$ and $S_b$ such that $S=S_a\cup S_b$ where the vertices in $S_a$ are adjacent to $G_a$ and the vertices in $S_b$ are adjacent to $G_b$. Also, for every vertex $v$ in $S_a$,  $d(v,S_b)=r-1$.
\end{remark}

\begin{lemma}\label{L: r-separator} Let $G$ be a uniform graph and $r\geq 2$.
Given any geodesic $[ab]$ with $d(a,b)>r$ and two vertices $v_1,v_2\in [ab]$ distinct from $a,b$ with $d(v_1,v_2)=r-1$, then there is a minimal $ab$-$r$-separator containing $\{v_1,v_2\}$.
\end{lemma}

\begin{proof} Suppose $[ab]$ is a geodesic with $d(a,b)>r$. Let us assume that $d(a,v_1)<d(a,v_2)$ and define $\varepsilon_1=d(a,v_1)$ and $\varepsilon_2=d(v_2,b)$. Since $G$ is uniform, for every vertex $v\in G$ the set $S(v,\varepsilon)$ is finite for every $\varepsilon\in \NN$. Let $S_0:=S(a,\varepsilon_1)\cup S(b,\varepsilon_2)$. It is immediate to check that $S_0$ is an $ab$-r-separator and $[ab]\cap S_0=\{v_1,v_2\}$. Since $S_0$ is finite, then there is a minimal subset $S\subset S_0$ which is also an $ab$-r-separator. Finally, since $[ab]\cap S_0=\{v_1,v_2\}$, $v_i\in S$ for $i=1,2$.
\end{proof}

\begin{theorem}\label{T1} Let $G$ be a uniform graph and $r\geq 2$. If every minimal vertex $r$-separator has diameter at most $m$ with $m\leq r$, then $G$ is $(r+\frac12)$-densely $(2r+2,r)$-chordal.
\end{theorem}

\begin{proof} Let $C$ be any cycle with $L(C)\geq 2r+2$ and let $x_1$ be any vertex in $C$. Then, consider two vertices $a,b$ in $C$ such that $d_C(a,b)= r+1$, $d_C(a,x_1)= 1$ and $d_C(x_1,b)= r$. Let $\gamma_1$ and $\gamma_2$ be the two independent paths joining $a$ and $b$ defined by $C$ and assume $x_1\in \gamma_1$. Consider $x_2\in \gamma_1$ with $x_2$ between $x_1$ and $b$ such that $d_C(x_1,x_2)=r-1$ (and $d_C(x_2,b)= 1$). 

If $\gamma_1$ is not a geodesic, then there is a shortcut with length at most $r$ and a shortcut vertex $v$ such that $d_C(x_1,v)< r$.

If $\gamma_1$ is a geodesic, by Lemma \ref{L: r-separator}, there exists a minimal $ab$-$r$-separator $S$ containing $x_1,x_2$. 
Then, there exist $y_1,y_2\in \gamma_2\cap S$, with $y_1$ between $a$ and $y_2$, such that $d_C(y_1,y_2)\geq r-1$, $d_C(a,y_1)\geq 1$ and $d_C(y_2,b)\geq 1$. 
Since $diam(S)\leq m$, then $d(x_1,y_2)\leq m$. 
Since $d_C(x_1,y_2)\geq r+1\geq m+1$, there is a shortcut $\sigma$ in $C$ joining $x_1$ and $y_2$ with $L(\sigma)\leq m\leq r$ and with an associated shortcut vertex $v$ such that $d_C(x_1,v)<m\leq r$.

Thus, for every vertex $x_1$, there is a shortcut vertex $v$ such that $d_C(x_1,v)<r$ and therefore, shortcut vertices define a $(r+\frac12)$-dense subset in $C$.
\end{proof}

\begin{theorem}\label{T2} Let $G$ be a uniform graph and $r\geq 2$. If for every minimal $ab$-$r$-separator $S$ either $S_a$ or $S_b$ has diameter at most $m$, 
then $G$ is $\varepsilon$-densely $(k,\frac{k}{2})$-chordal with $k=2m+2r+2$ and $\varepsilon=\max\{\frac{m+1}{2}+r,m+\frac{1}{2}\}$ .
\end{theorem}

\begin{proof} Let $C$ be any cycle with $L(C)\geq 2m+2r+2$ and $x_1$ be any vertex in $C$. Then, consider two vertices $a,b$ in $C$ such that $d_C(a,b)= m+r+1$, $d_C(a,x_1)= \frac{m+1}{2}$ and 
$d_C(x_1,b)= \frac{m+1}{2}+r$ if $m$ is odd, and $d_C(a,x_1)= \frac{m}{2}+1$ and 
$d_C(x_1,b)= \frac{m}{2}+r$ if $m$ is even.  
Let $\gamma_1$ and $\gamma_2$ be the two independent paths joining $a$ and $b$ defined by $C$ and assume $x_1\in \gamma_1$. Consider $x_2\in \gamma_1$ with $x_2$ between $x_1$ and $b$ such that $d_C(x_1,x_2)=r-1$ (and therefore, $d_C(x_2,b)\geq \frac{m}{2}+1>\frac{m}{2}$).

If $\gamma_1$ is not a geodesic, then there is a shortcut with length at most $m+r+1$ and a shortcut vertex $v$ such that 
$d_C(x_1,v)< \frac{m}{2}+r$.

If $\gamma_1$ is a geodesic, consider $S$ the minimal $ab$-$r$-separator containing $x_1,x_2$  built in the proof of Lemma \ref{L: r-separator}.
and let us assume, without loss of generality, that $S_a$ has diameter at most $m$.  Then, by construction, there exists $y_1\in \gamma_2\cap S$ such that 
$d_{\g_2}(a,y_1)\geq d(a,y_1)=d(a,x_1)\geq \frac{m+1}{2}$. Since $diam(S_a)\leq m$, then $d(x_1,y_1)\leq m$. However, $d_C(x_1,y_1)\geq \min\{m+1,d_{\g_1}(x_1,b)+d(b,y_1)\}= m+1$ and therefore, there is a shortcut $\sigma$ in $C$ joining $x_1$ and $y_1$ with $L(\sigma)\leq m$. Moreover, there is a shortcut vertex $v$ such that $d_C(x_1,v)< m$.

Thus, for every vertex $x_1$, there is a shortcut vertex $v$ with 
$d_C(x_1,v)<\min\{\frac{m}{2}+r,m\}$  and therefore, shortcut vertices define an 
$\varepsilon$-dense subset in $C$ with $\varepsilon=\max\{\frac{m+1}{2}+r,m+\frac{1}{2}\}$.
\end{proof}

Then, from Theorems \ref{Th: ch-BP} and \ref{T2}, we can obtain immediately the following:

\begin{corollary} Let $G$ be a uniform graph and $r\geq 2$. If for every minimal $ab$-$r$-separator $S$ either $S_a$ or $S_b$ has diameter at most $m$, then $G$ satisfies (BP), i.e., $G$ is quasi-isometric to a tree.
\end{corollary}

Also, from theorems \ref{Th: hyperbolic}, \ref{T1} and \ref{T2} we obtain:

\begin{corollary} Let $G$ be a uniform graph and $r\geq 2$. If every minimal vertex $r$-separator has diameter at most $m$ with $m\leq r$, then $G$ is $\delta$-hyperbolic. Moreover, 
$\delta(G)\leq 2r+\frac12$.
\end{corollary}

\begin{corollary} Let $G$ be a uniform graph and $r\geq 2$. If for every minimal $ab$-$r$-separator $S$ either $S_a$ or $S_b$ has diameter at most $m$, then $G$ is $\delta$-hyperbolic. Moreover, $\delta(G)\leq \max\{\frac{3m+3}{2}+2r,2m+r+\frac{3}{2}\}\}$.
\end{corollary}



\section{Neighbor separators}\label{S:5}

Given a set $S$ in a graph $G$, let $N_r(S):=\{x\in G \, | \, d(x,S)\leq r\}$. 

\begin{definition} Given two vertices $a,b$ in a graph $G=(V,E)$  and some $r\in \mathbb{N}$, a set $S\subset V$  is an $ab$-$N_r$-\emph{separator}  if $a$ and $b$ are in different components of $G\setminus N_r(S)$. $S$ is an $ab$-\emph{neighbor separator} if it is an $ab$-$N_r$-separator for some $r$.
\end{definition}

Notice that an $ab$-separator is just an $ab$-$N_0$-separator.

\begin{theorem} $G$ satisfies (BP) if and only if there is a constant $\Delta''>0$ such that for every pair of vertices $a,b$ with $d(a,b)\geq 2\Delta''+2$  and any geodesic $[ab]$, there exists a vertex 
$c\in [ab]$ which is an $ab$-$N_{\Delta''}$-separator.
\end{theorem}

\begin{proof} The only if part follows trivially from Proposition \ref{P: BP graphs}.

Suppose that for every pair of vertices $a,b$ with $d(a,b)\geq 2\Delta''+2$ and any geodesic $[ab]$, there exists a point $c\in [ab]$ which is an $ab$-$N_{\Delta''}$-separator. Consider any pair of vertices $x,y$ in $G$, any geodesic $[xy]$ and the midpoint $z$ in $[xy]$. 

If $d(x,y)\leq 2\Delta''+1$ then (BP) is trivially satisfied on $x,y$ for any 
$\D'\geq \D''+\frac12$. 

If $d(x,y)\geq 2\Delta''+2$, by hypothesis there is some vertex $z_1\in [xy]$ such that $N_{\D''}(z_1)$ is an $xy$-$N_{\Delta''}$-separator. If $d(z,z_1)\leq \D''$ then it follows that every $xy$-path intersects $N_{\D''}(z_1)\subset N_{2\D''}(z)$ and $G$ satisfies (BP) on the vertices for $\D'=2\D''$. If $d(z,z_1)> \D''$ then we repeat the process with the part of the geodesic, $[xz_1]$ or $[z_1y]$, containing $z$. Let us assume, without loss of generality, that $z\in [xz_1]$. Since $d(z,z_1)> \D''$ and $d(x,z)>\D''$, there is some point $z_2\in [xz_1]$ which is an $xz_1$-$N_{\Delta''}$-separator. Since there is a $z_1y$-path in $G\setminus N_{\D''}(z_2)$, $z_2$ is also an $xy$-$N_{\Delta''}$-separator. If $d(z,z_2)\leq \D''$ we are done. Otherwise, we repeat the process until we obtain some point $z_k\in [xy]$ which is an $xy$-$N_{\Delta''}$-separator and such that $d(z,z_k)\leq \D''$. Therefore, $G$ satisfies (BP) on the vertices for $\D'=2\D''$.  

Thus, by Proposition \ref{P: BP graphs}, $G$ satisfies (BP) with $\D=2\D''+\frac32$. 
\end{proof}

\begin{corollary} $G$ is quasi-isometric to a tree if and only if there is a constant $\Delta''>0$ such that for every pair of vertices $a,b$ with $d(a,b)>\Delta''$ and any geodesic $[ab]$, there exists a vertex $c\in [ab]$ which is an $ab$-$N_{\Delta''}$-separator.
\end{corollary}

\begin{proposition}\label{Prop: N-sep} If $G$ is $(k,1)$-chordal, then for every pair of vertices $a,b$, any geodesic $[ab]$ with $d(a,b)\geq \frac{k}{2}+2$ and every pair of vertices $a',b'\in [ab]$ with $d(\{a',b'\},\{a,b\})\geq 2$ and such that $d(a',b')\geq \frac{k}{2}-2$, $[a'b']\subset [ab]$ is an $ab$-$N_1$-separator. In particular, for every pair of vertices $a,b$ in $G$ with $d(a,b)\geq \frac{k}{2}+2$ there is a geodesic $\sigma$ of length $\frac{k}{2}-2$ or $\frac{k-3}{2}$ such that $\sigma$ is an $ab$-$N_1$-separator.
\end{proposition}

\begin{proof} Consider any geodesic $[ab]$ in $G$ with $d(a,b)\geq \frac{k}{2}+2$ and any pair of vertices $a',b'\in [ab]$ with $d(\{a',b'\},\{a,b\})\geq 2$ such that $d(a',b')\geq \frac{k}{2}-2$. Let $a''$ be the vertex in $[aa']\subset [ab]$ adjacent to $a'$  and $b''$ be the vertex in $[b'b]\subset [ab]$ adjacent to $b'$. Therefore, $d(a'',b'')\geq \frac{k}{2}$. Suppose that $a$ and $b$ are in the same connected component, $A$, of $G\setminus N_1([a'b'])$. Clearly, $a''$ and $b''$ are adjacent to $A$. Let $\gamma$ be a path of minimal length joining $a''$ and $b''$ in the subgraph induced by $A\cup \{a'',b''\}$. Therefore, $[a''b'']\cup \gamma$ defines a cycle, $C$, of length at least $k$. Since $G$ is $(k,1)$-chordal, then there is an edge joining two non-adjacent vertices in $C$. Since $[a''b'']$ is geodesic and $\gamma$ has minimal length, the edge must join a vertex, 
$v \in \gamma$ to a vertex in $[a'b']$. Therefore, $v\in N_1([a'b'])\cap A$ leading to contradiction.  
\end{proof}

\begin{definition} A path $\g$ in a graph $G$ is \emph{chordal} if it has no 1-shortcuts in $G$.
\end{definition}

\begin{proposition}\label{Prop: N-sep_chordal} If $G$ is $(k,1)$-chordal, then for every chordal $ab$-path $\s$  with $L(\s)\geq k$ 
and every pair of vertices $a',b'\in \s$ with $d_\s(\{a',b'\},\{a,b\})\geq 2$ and such that $d_\s(a',b')\geq k-4$, then  the restriction of $\s$ joining $a'$ and $b'$, $\s'$, is an $ab$-$N_1$-separator. In particular, for every pair of vertices $a,b$ in $G$ joined by a chordal path with length at least $k$ there is a chordal path $\gamma'$ of length $k-4$  such that $\gamma'$ is an $ab$-$N_1$-separator.
\end{proposition}

\begin{proof} Consider any chordal path $\sigma$ in $G$ with endpoints $a,b$ and $L(\sigma)\geq k$. Consider any pair of vertices $a',b'\in [ab]$ with $d_\s(\{a',b'\},\{a,b\})\geq 2$ such that $d_\s(a',b')\geq k-4$ and let $\s'=[a'b']\subset [ab]$. Let $a''$ be the vertex in $\sigma$ adjacent to $a'$ closer in $\sigma$ to $a$ and $b''$ be the vertex in $\sigma$ adjacent to $b'$ closer in $\sigma$ to $b$. Therefore, if $\sigma''$ is the restriction of $\sigma$ joining $a''$ and $b''$, then $L(\sigma'')\geq k-2$. Suppose that $a$ and $b$ are in the same connected component, $A$, of 
$G\setminus N_1(\sigma')$. Clearly, $a''$ and $b''$ are adjacent to $A$. Let $\gamma$ be a path of minimal length joining $a''$ and $b''$ in the subgraph induced by $A\cup \{a'',b''\}$. Therefore, $\sigma''\cup \gamma$ defines a cycle, $C$, of length at least $k$. Since $G$ is $(k,1)$-chordal, then there is an edge joining two non-adjacent vertices in $C$. Since $\sigma''$ is chordal and $\gamma$ has minimal length, the edge must join a vertex, $v \in \gamma$ to a vertex in $\sigma'$. Therefore, $v\in N_1(\sigma')\cap A$ leading to contradiction.  
\end{proof}

\begin{proposition} If a graph $G$ satisfies that for some $k,m\in \mathbb{N}$ with $k\geq 4m$, for every geodesic $[ab]$ with $d(a,b)\geq k+2$  and for every pair of vertices $a',b'\in [ab]$ with $d(\{a',b'\},\{a,b\})\geq m+1$ and such that $d(a',b')\geq k-2m$, $[a'b']\subset [ab]$ is an $ab$-$N_m$-separator, then $G$ is $(\frac{k}{2}+2)$-densely $(2k+4, k+1)$-chordal.
\end{proposition}

\begin{proof} Let $C$ be any cycle with $L(C)\geq 2k+4$. Let $v$ by any vertex in $C$ and $a,b$ two vertices in $C$ such that $d_C(a,v)=\left\lfloor\frac{k}{2} \right\rfloor+1$ and $d_C(v,b)=\left\lceil \frac{k}{2}\right\rceil +1$, and therefore, $d_C(a,b)=k+2$. Let $\gamma_1, \gamma_2$ be  the two $ab$-paths defined by the cycle and let us assume that $v\in \g_1$ (and therefore, $L(\gamma_1)\leq L(\gamma_2)$). If there is a shortcut in $\g_1$, then there is a shortcut in $C$ with length at most $k+1$ and with a shortcut vertex $z$ such that $d_C(v,z)<\frac{k}{2}+2$. If 
there is no shortcut in $\g_1$, then $\gamma_1$ is a geodesic with $d(a,b)= k+2$. 
Thus, let $a',b'\in \gamma_1$ with $d(a,a')= m+1=d(b',b)$ and $d(a',b')= k-2m$. Therefore $[a'b']\subset \gamma_1$ is an $ab$-$N_m$-separator. In particular, there is some vertex $w$ in $\gamma_2\setminus \{a,b\}$ such that $w\in N_m([a'b'])$, defining a shortcut in $C$ with length at most $m$ and with a shortcut vertex $z$ such that $d_C(v,z)<\frac{k}{2}+1$.
\end{proof}

\begin{corollary} If a graph $G$ satisfies that for some $k,m\in \mathbb{N}$ with $k\geq 4m$, for every geodesic $[ab]$ with $d(a,b)\geq k+2$  and for every pair of vertices $a',b'\in [ab]$ with $d(\{a',b'\},\{a,b\})\geq m+1$ and such that $d(a',b')\geq k-2m$, $[a'b']\subset [ab]$ is an $ab$-$N_m$-separator, then $G$ is quasi-isometric to a tree.
\end{corollary}

\section{Neighbor obstructors}\label{S:6}

\begin{definition} Given two vertices $a,b$ in a graph $G=(V,E)$  and some $r\in \mathbb{N}$, a set $S\subset V$  is $ab$-$N_r$-obstructing if for every geodesic $\gamma$ joining $a$ and $b$, $\gamma \cap N_r(S)\neq \emptyset$. 
\end{definition}

Given any metric space $(X,d)$ and any pair of subsets $A,B \subset X$, let us recall that the Hausdorff metric, $d_H$, induced by $d$  is 
$$d_H(A_1,A_2):=max\{\sup_{x\in A_1}\{d(x,A_2)\},\sup_{y\in A_2}\{d(y,A_1)\}\},$$ 
or equivalently,
$$d_H(A_1,A_2):= \inf\{\varepsilon>0 \ | \ A_1\subset
B(A_2,\varepsilon) \\ \mbox{ y } A_2\subset B(A_1,\varepsilon)\}.$$

\begin{definition} In a geodesic metric space $(X,d)$, we say that \emph{geodesics are stable} if and only if there is a constant $R\geq 0$ such that given two points $x,y\in X$ and any geodesic $[xy]$, then every geodesic $\s$ joining $x$ to $y$ satisfies that $d_H(\s,[xy])\leq R$.
\end{definition}

It is well known that if $X$ is a hyperbolic space, then quasi-geodesics are stable. See, for example, Theorem III.1.7 in \cite{B-H}. In particular, geodesics are stable in hyperbolic geodesic spaces.

\smallskip

Let $\mathcal{B}$ be the family of cycles that are bigons.

\begin{theorem}\label{T: stability}  Given a graph $G$, geodesics are stable if and only if there exist constants $\varepsilon>0$ and $k,m\in \NN$ such that $G$ is $\varepsilon$-densely $(k,m)$-chordal on $\mathcal{B}$.
\end{theorem}

\begin{proof} Suppose that $G$ is $\varepsilon$-densely $(k,m)$-chordal on $\mathcal{B}$. Consider any pair of points $x,y$ and any pair of geodesics, $\s_1,\s_2$, joining them. Then, for any point 
$z\in \s_1$, either $z\in \s_1\cap \s_2$ or there is a cycle $C\subset \s_1\cup \s_2$ with $z\in C$. If $L(C)<k$, then $d(z,\s_2)<\frac{k}{2}$. If $L(C)\geq k$, then either $d_C(z,\s_2)<\varepsilon$ or there is an $m$-shortcut in $C$ with a shortcut vertex $v$ such that $d_C(z,v)<\varepsilon$ and since $\s_1$ is geodesic, $d(v,\s_2)\leq m$. Thus, if $R=\max\{\frac{k}{2}, \varepsilon+m\}$, $d(z,\s_2)< R$ in any case. Hence, $\s_1\subset N_{R}(\s_2)$. The same argument proves that $\s_2\subset N_{R}(\s_1)$ and therefore, $d_H(\s_1,\s_2)\leq R$. 

Suppose that geodesics are stable with constant $R$. Consider any pair of points $x,y$ with 
$d(x,y)\geq 2R+2$ and two $xy$-geodesics $\s_1,\s_2$ such that $\s_1\cup \s_2$ defines a cycle $C$. 
Then, for any point $z\in \s_1$ (respectively with $\s_2$) such that $d_C(z,\s_2)>R$ (resp. 
$d_C(z,\s_1)>R$), since $d_H(\s_1,\s_2)\leq R$, $d(z,\s_2)\leq R$ (resp. $d(z,\s_1)\leq R$) and there is an strict $R$-shortcut in $C$ with a shortcut vertex $v$ such that $d_C(v,z)<R$. Thus, shortcut vertices are $(2R+1)$-dense in $C$ and $G$ is $(2R+1)$-densely $(4R+4,R)$-chordal on $\mathcal{B}$. 
\end{proof}

\begin{definition} In a graph $G$, we say that \emph{geodesics between vertices are stable} if and only if there is a constant $R\geq 0$ such that given two vertices $a,b\in G$ and any geodesic $[ab]$, then every geodesic $\s$ joining $a$ to $b$ satisfies that $d_H(\s,[ab])\leq R$.
\end{definition}

\begin{proposition}\label{P: stable} Given a graph $G$, geodesics between vertices are stable   if and only if there is some constant $k\in \NN$ so that for every pair of vertices $a,b$ with $d(a,b)\geq 2k+2$, every geodesic $[ab]$ and every vertex $v\in [ab]$ such that $d(v,\{a,b\})> k$, then $v$ is an $ab$-$N_k$-obstructing vertex. 
\end{proposition}

\begin{proof} Suppose that geodesics between vertices are stable with constant $R$. Then, given any two vertices $a,b\in G$ with $d(a,b)\geq 2R+2$ and any geodesic $[ab]$, every geodesic $\s$ joining $a$ to $b$ satisfies that $d_H(\s,[ab])\leq R$. Thus, for every vertex $v\in [ab]$ there is some vertex $w\in \s$ such that $d(v,w)\leq R$.  Suppose $v\in [ab]$ with $d(v,\{a,b\})> R$. 
Hence, $v$ is an $ab$-$N_R$-obstructing vertex.  

Now suppose that for every pair of vertices $a,b$ with $d(a,b)\geq 2k+2$, every geodesic $[ab]$ and every vertex $v\in [ab]$ with $d(v,\{a,b\})> k$, then $v$ is an $ab$-$N_k$-obstructing vertex.  Consider any pair of vertices $a,b\in G$ and any pair of $ab$-geodesics $\s_1,\s_2$. 
If $d(a,b)< 2k+2$, then it is trivial to check that $d_H(\s_1,\s_2)< k+1$. 
Suppose $d(x,y)\geq 2k+2$. Then, for every vertex $v\in \s_1$ such that $d(v,\{a,b\})> k$,
$\s_2\cap N_k(v)\neq \emptyset$. Therefore, it follows immediately that $\s_1\subset N_{k+1/2}(\s_2)$. The same argument proves that $\s_2\subset N_{k+1/2}(\s_1)$ and therefore, $d_H(\s_1,\s_2)< k+1$. 
\end{proof}

Let $\mathcal{B}_0$ be the family of cycles that are bigons defined by two geodesics between vertices.

\begin{proposition}\label{Prop: N(r)-obs} If $G$ is $\frac{k}{4}$-densely $(k,m)$-chordal on $\mathcal{B}_0$, then for every pair of vertices $a,b$ with $d(a,b)\geq \frac{k}{2}+4$, every geodesic $[ab]$ and every vertex $v_0$ such that $d(v_0,\{a,b\})\geq \frac{k}{4}+1$, $v_0$ is an $ab$-$N_k$-obstructing vertex. In particular, $[ab]$ contains an $ab$-$N_k$-obstructing vertex.
\end{proposition}

\begin{proof} Consider any pair of vertices $a,b$ with $d(a,b)\geq \frac{k}{2}+4$, any geodesic 
$[ab]$ and any vertex $v_0\in [ab]$ with $d(v_0,\{a,b\})\geq \frac{k}{4}+1$. Let $a'$ be the vertex in $[av_0]\subset [ab]$ with $d(a',v_0)=\left\lceil \frac{k}{4}\right\rceil$  and $b'$ be the vertex in $[v_0b]\subset [ab]$ with $d(v_0,b')=\left\lceil \frac{k}{4}\right\rceil$. Therefore, $d(a',b')\geq \frac{k}{2}$, $a'\neq a$ and $b'\neq b$.

If there is no geodesic joining $a$ to $b$ disjoint from $N_{k/4}(v_0)$ and we are done.

Suppose there is some geodesic $\gamma_0$   joining $a$ to $b$ such that 
$\gamma_0 \cap N_{k/4}(v_0)=\emptyset$. Then, $[ab]\cup \gamma_0$ contains a cycle $C$ (with possibly 
$C=[ab]\cup \gamma_0$) composed by two geodesics $\gamma_1=[a''b'']$ with $[a'b'] \subset [a''b'']\subset [ab]$ and $\gamma_2\subset \gamma_0$ joining also $a''$ to $b''$. Clearly, $L(C)\geq k$. Since $G$ is $\frac{k}{4}$-densely $(k,m)$-chordal on $\mathcal{B}_0$, then there is a strict shortcut $\sigma$ with $L(\sigma)\leq m$ joining two  vertices in $C$ with a shortcut vertex $v_1$ in $N_{k/4}(v_0)$. 
Also, since $\gamma_1$ and $\gamma_2$ are geodesics, then $\sigma$ joins $v_1$ to a vertex $v_2$ in 
$\gamma_2\subset \gamma_0$. 
Therefore, $d(v_2,v_0)\leq m+ \frac{k}{4}< k$ (see Remark \ref{k 4}) and $\gamma_0\cap N_k(v_0)\neq \emptyset$. 
\end{proof}

\begin{theorem}\label{T: stability 2}  Given a graph $G$, geodesics between vertices are stable if and only if there exist constants $\varepsilon>0$ and $k,m\in \NN$ such that $G$ is $\varepsilon$-densely $(k,m)$-chordal on $\mathcal{B}_0$.
\end{theorem}

\begin{proof} Suppose that $G$ is $\varepsilon$-densely $(k,m)$-chordal on $\mathcal{B}_0$. By  
Proposition \ref{Prop: N(r)-obs}, if $k'=\max\{4\varepsilon,k\}$, then for every pair of vertices $a,b$ with $d(a,b)\geq \frac{k'}{2}+4$, every geodesic $[ab]$ and every vertex $v_0$ such that $d(v_0,\{a,b\})\geq \frac{k'}{4}+1$, $v_0$ is an $ab$-$N_{k'}$-obstructing vertex. 
Thus, by Proposition \ref{P: stable}, geodesics are stable with constant $R=\frac{k'}{4}+2$.

Let us suppose that geodesics between vertices are stable with constant $R$. 
Let $a,b$ be two vertices with $d(a,b)\geq 2R+2$ and $C$ be a cycle which is a bigon defined by two $ab$-geodesics, $\s_1,\s_2$. Therefore, $L(C)\geq 4R+4$. Consider any vertex $v\in \s_1$ (respectively, $\s_2$) such that $d(v,\{a,b\})>R$. Then, since geodesics between vertices are stable with parameter $R$, $v\in N_R(\s_2)$ (respectively, $\s_1$) and there is a strict $R$-shortcut in $C$ with an associated shortcut vertex $w$ such that $d_C(v,w)<R$. Therefore, shortcut vertices are $(2R+1)$-dense in $C$ and $G$ is $(2R+1)$-densely $(4R+4,R)$-chordal on $\mathcal{B}_0$.
\end{proof}

The following example shows that having stable geodesics between vertices does not imply that geodesics are stable. 

\begin{example} Consider the family of odd cycles $\{C_{2k+1} \, : \, k\in \NN \}$ and suppose we fix a vertex $v_k$ in each cycle and we define a connected graph $G$ identifying the family $\{v_k \, : \, k\in \NN\}$ as a single vertex $v$. Notice that in $G$ geodesics between vertices are unique. If two vertices belong to the same cycle $C_{2k+1}$, then the geodesic is contained in the cycle and it is clearly unique. Otherwise, the geodesic is the union of the two (unique) shortest paths joining the vertices to $v$. Thus, geodesics between vertices are stable with constant 0. 

Let $m_k$ be the midpoint of an edge in $C_{2k+1}$ such that $d(m_k,v)=k+\frac{1}{2}$. 
Then, $C_{2k+1}$ is a bigon in $G$ defined by two geodesics, $\s_1,\s_2$  joining $m_k$ to $v$ and 
$d_H(\s_1,\s_2)=\frac{k}{2}+\frac{1}{4}$ with $k$ arbitrarily large. 
\end{example}

\begin{remark} Notice that the same property that characterizes being quasi-isometric to a tree (Corollary \ref{C: ch-BP}), also characterizes being hyperbolic, when restricted to triangles (Theorem \ref{Th: carac2}), having stable geodesics, when restricted to bigons (Theorem \ref{T: stability}) and having stable geodesics between vertices, when restricted to bigons between vertices (Theorem \ref{T: stability 2}). 
\end{remark}

\smallskip

The proof of Proposition \ref{Prop: N(r)-obs} can be adapted to prove also the following:

\begin{proposition}\label{Prop: N(r)-obs 2} If $G$ is $(\frac{k}{4}-m)$-densely $(k,m)$-chordal on $\mathcal{B}_0$ with $k> 4m$, then for every geodesic $[ab]$ with $d(a,b)\geq \frac{k}{2}+2$ and every pair of vertices $a',b'\in [ab]$ with $d(\{a',b'\},\{a,b\})\geq m+1$ and such that $d(a',b')\geq \frac{k}{2}-2m$, $[a'b']\subset [ab]$ is an $ab$-$N_m$-obstructing set. In particular, for every pair of vertices $a,b$ in $G$ with $d(a,b)\geq \frac{k}{2}+2$ there is a geodesic $\sigma$ of length $\frac{k}{2}-2m$ or $\frac{k+1}{2}-2m$ such that $\sigma$ is $ab$-$N_m$-obstructing.
\end{proposition}

\begin{proof} Consider any geodesic $[ab]$ with $d(a,b)\geq \frac{k}{2}+2$ and any pair of vertices $a',b'\in [ab]$ with $d(\{a',b'\},\{a,b\})\geq m+1$ and $d(a',b')\geq \frac{k}{2}-2m$. Let $a''$ be the vertex in $[aa']\subset [ab]$ with $d(a',a'')=m$  and $b''$ be the vertex in $[b'b]\subset [ab]$ with $d(b',b'')=m$. Therefore, $d(a'',b'')\geq \frac{k}{2}$.

Suppose that there is some geodesic $\gamma_0$   joining $a$ and $b$ such that $\gamma \cap N_m([a'b'])=\emptyset$. Then, $[ab]\cup \gamma_0$ contains a cycle $C$ composed by two geodesics: $\gamma_1$ with $[a''b''] \subset \g_1 \subset [ab]$ and $\gamma_2\subset \gamma_0$. 
Clearly, $L(C)\geq k$. 

Consider the midpoint $c$ in $[a'b']$. Since $G$ is $(\frac{k}{4}-m)$-densely $(k,m)$-chordal on 
$\mathcal{B}_0$, then there is a strict shortcut $\sigma$ with $L(\sigma)\leq m$ joining two  vertices in $C$ with a shortcut vertex $v_1$ such that $d_C(v_1,c)\leq \frac{k}{4}-m$, and hence $v_1\in [a'b']$.  Also, since $\gamma_1$ and $\gamma_2$ are geodesics, then $\sigma$ joins $v_1$ to a vertex, $v_2$, in $\gamma_2\subset \gamma_0$. 
Therefore, $d(v_2,[a'b'])\leq m$ and $\gamma_0\cap N_m([a'b'])\neq \emptyset$ leading to contradiction. 
\end{proof}


\end{document}